\documentclass[11pt]{amsart} 

\usepackage[margin=2.9cm]{geometry}             
\usepackage{color,ulem}
\usepackage{graphicx}
\usepackage{amssymb, enumerate,bm}
\usepackage[matrix,arrow,ps]{xy}
\usepackage{epstopdf, amsmath} 
\usepackage{paralist}
\newcommand{\C}{\mathbb C}

\newcommand{\R}{\mathbb R}

\newcommand{\transp}{\,^t}

\newcommand{\vnorm}[1]{{\| #1 \|}}

\DeclareMathOperator{\spanc}{span}
\newtheorem{theo}{Theorem}[section]
\newtheorem{lemma}[theo]{Lemma}

\newtheorem{prop}[theo]{Proposition}

\theoremstyle{remark}
\newtheorem{remark}[theo]{Remark}
\theoremstyle{example}
\newtheorem{example}[theo]{Example}
\theoremstyle{definition}
\newtheorem{defi}[theo]{Definition}
\numberwithin{equation}{section}

\begin{document}

\begin{abstract}
We study the family of generalized stationary discs attached to a Levi degenerate submanifold $M$ of codimension $d$ in $\C^{n+d}$. We show, under suitable geometric assumptions on $M$, that this family forms a finite dimensional real submanifold of the Banach space of analytic discs.    
\end{abstract} 
\thanks{This work was supported by the Center for Advanced Mathematical Sciences, by the Faculty of Arts and Sciences, and by a URB grant from the American University of Beirut.}

\author[Al Masri, Bertrand, Meylan, Oueidat, Zoghaib]{Mohammad Tarek Al Masri, Florian Bertrand, Francine Meylan, Lea Oueidat, Hadi Zoghaib}
\title[Generalized Stationary discs attached to degenerate submanifolds in $\C^N$.]{Generalized Stationary discs attached to degenerate submanifolds in $\C^N$.}

\subjclass[2010]{}

\thanks{}

\maketitle


\section*{Introduction} 

Let $M \subset \C^{n+d}$ be a  germ of a finitely smooth  generic real submanifold  of codimension $d$ at $p,$  and let $Aut(M, p)$ be its stability group, that is, the set of germs of biholomorphisms mapping $M$ into itself and fixing $p$. Initiated by Lempert in \cite{le}, the theory of   stationary discs  attached to $M$  developed, for instance   in \cite{be-bl, be-bl-me, be-de1, be-de-la, be-me, bl, hu, tu, tu3}, turns out to be  a powerful tool to detect finite jet determination  for elements of  $Aut(M, p)$.  More precisely,  in    the case of  $\mathcal{C}^4$  Levi nondegenerate  generic real submanifolds, the study of stationary  discs   yields algebraic conditions that force   2-jet determination for elements of  $Aut(M, p)$ and more generally for germs of  CR automorphisms of $M$
 of class $\mathcal{C}^3$ and fixing $p$ (see \cite{be-bl,be-bl-me,be-me,tu3}). In the  case of Levi degenerate real hypersurfaces, by constructing {\it generalized} (or higher order) stationary discs, 
 finite jet determination for finitely smooth CR diffeomorphisms of a class of weakly pseudoconvex hypersurfaces of finite type in $\C^2$ 
  has been obtained in 
  \cite{be-de1}, while  finite jet determination  for finitely smooth CR diffeomorphisms  of a class of  generically  Levi nondegenerate 
   hypersurfaces of finite type in $\C^{n+1}$ has been established in \cite{be-de-la}.

   In the recent work \cite{al-be-mc-ou-zo},  the authors  start building a theory of    generalized stationary discs for Levi degenerate  generic real submanifold by considering   perturbations of decoupled submanifolds  in $\C^{4}.$ In the present  paper, we address this question  for Levi degenerate  submanifolds $M$  of codimension $d$ whose models are {\it rigid}, that is,  given   in suitable holomorphic   coordinates   $(z,w)\in \C^{n+d}$,  referred to as {\it canonical coordinates} in Theorem 4.3.2 \cite{ber},
  at $p=0$ by 
  \begin{equation*} 
  M= \{ \Re e  w= P(z, \bar z)+ {\rm higher \ order \ terms} \},
  \end{equation*}  where $P=(P_1,\ldots, P_d )\neq 0$,  and $P_\ell$, $\ell=1,\ldots,d,$ is a homogeneous polynomial, with no pluriharmonic terms, of degree $D_\ell$ or $P_\ell\equiv 0$.  We observe that a class of  {\it finite type} submanifolds, in the sense of Kohn and Bloom-Graham (see \cite{ber,bl-me,BG,K}), may be expressed in this form. In particular, any hypersurface of finite type in $\C^{n+1}$ is of this form.    
  In this paper, we focus on the model, denoted by 
  $M_H= \{ \Re e w= P(z, \bar z)\}$, and consider $M$ as a perturbation of $M_H$.  
The construction  of generalized stationary  discs for perturbations of a Levi degenerate model $M_H$ is mainly obstructed  by the fact that the conormal  bundle of $M_H$ is no longer totally real (see \cite{tu}). In the hypersurface case studied in \cite{be-de-la}, 
 the idea is then to consider models, referred to as {\it admissible}, that admit an initial generalized stationary disc passing through the degeneracy locus of $M_H$  at a {\it single point with a prescribed order}. The analogous condition, also called {\it admissibilty}, is introduced in Definition \ref{defiadmi} and plays a crucial role in our approach. 
 Our main result, Theorem \ref{theodiscs}, concerns the construction of a finite dimensional submanifold  of stationary discs for  a class of  finitely smooth perturbations $M$ of an admissible model $M_H$. We point out that its dimension is explicit and depends ${\it only}$ on the geometry of the model $M_H$.  
 An interesting application 
 of Theorem \ref{theodiscs} is when $M_H$ is included in the product of a Levi nondegenerate, or more generally admissible, hypersurface in $\C^{n+1}$ and $\C^{d-1}$.   
 The proof of Theorem \ref{theodiscs}  relies on the study of a singular Riemann-Hilbert type problem by means of an adapted implicit function theorem and techniques developed by Forstenri\v{c} \cite{fo}, Globevnik \cite{gl1,gl2}, and Della Sala and the second author \cite{be-de2}.

\vspace{0.5cm}
 
The paper is organized as follows. In Section  \ref{coucou}, we introduce  the notion of admissibility for a model submanifold $M_H$ and provide examples of 
such models.  In Section \ref{coucou2}, we recall the different notions needed in the rest of the paper, that is,
  the Banach spaces of analytics discs with constraints, the notion of generalized stationary discs and the corresponding Riemann-Hilbert problem. Finally, we state and prove our main result  in Section \ref{coucou3} and take this opportunity to discuss a geometric property enjoyed by the family of discs constructed in Theorem \ref{theodiscs},  and open directions related to the jet determination of CR mappings.

\section{Admissible degenerate submanifolds of  $\C^{n+d}$}{\label{coucou}}
The unit disc in $\C$ is denoted by $\Delta=\{\zeta \in \C \ | \ |\zeta|<1\}$. We denote  its boundary by $\partial \Delta$.
Let $M \subset \C^{n+d}=\C^n_z\times \C^{d}_w$ be a finitely smooth  real submanifold of real codimension $d$ through $0$ given locally by 
\begin{equation}\label{eqdeg}
\begin{cases}
r_1=\Re e  w_1- P_1(z,\overline{z})+ O(D_1+1)=0\\
 \vdots \\
r_d=\Re e  w_d - P_d(z,\overline{z})+ O(D_d+1)=0
\end{cases}
\end{equation}
where $P_\ell$, $\ell=1,\ldots,d$, is a real homogeneous polynomial of  degree $D_\ell$ written as  
$$P_\ell(z,\overline z)=\sum_{\substack{|I| + |J| = D_\ell \\ D_\ell-k_\ell\leq |I| \leq k_\ell}}\alpha^\ell_{IJ} z^I \overline z^J$$
with  $\alpha_{IJ} = \overline \alpha_{JI}$. Here $I=(i_1,\ldots,i_n)$ and $J=(j_1,\ldots,j_n)$ are multi-indices.
We choose 
$$\frac{D_\ell}{2}\leq k_\ell\leq D_\ell-1$$ in such a way that there exists $(\tilde I, \tilde J)$ with $|\tilde J |=k_\ell$ such that $\alpha_{\tilde I \tilde J}\neq 0$. Moreover, we assume that $2\leq D_1\leq D_\ell\leq D_d$, $\ell=1,\ldots,d$. Note that $P_2,\ldots,P_d$ may be identically equal to zero. We write $M=\{r=0\}$ where $r:=(r_1,\ldots,r_d)$. We also set $k_0:=\max\{k_1,\ldots,k_d\}$  and point out that this integer will play an important role in the rest of the paper.  

 We associate with $M$ its model submanifold $M_H$ given by
\begin{equation}\label{eqmod}
\begin{cases}
\rho_1=\Re e  w_1- P_1(z,\overline{z})=0\\
\vdots\\ 
\rho_d=\Re e  w_d - P_d(z,\overline{z})=0
\end{cases}
\end{equation}
We set $\rho:=(\rho_1,\ldots, \rho_d)$ and then write $M_H=\{\rho=0\}$.   

\vspace{0.5cm}

Assume that the model $M_H$ is a  quadric,  that is, with 
\begin{equation}\label{eqA_j}
P_\ell(z,\overline{z})=\transp \overline{z} A_\ell z
\end{equation} for $\ell=1,\ldots,d$ where $A_\ell$ is a Hermitian $n\times n$ matrix. In this setting,  
$M_H$ is {\it strongly Levi nondegenerate} if there exists an invertible real linear combination of the $A_\ell$'s. This condition is crucial in the construction of stationary discs attached to the quadric $M_H$ and its perturbations.  
For degenerate models, we then introduce an  invariant notion which is the analogous to the {\it strong Levi nondegeneracy}.

We define, for any $d$ real numbers $c_1,\ldots,c_d \in \R$ and $V \in \C^n\setminus\{0\}$, functions $Q_{i\overline j}, S_{\bar{i} \bar{j}}: \partial \Delta \to \C$, $i,j=1,\ldots,n$, as follows
\begin{equation}\label{eqQS}
\left\{
\begin{array}{lll} 
\zeta^{k_0}\sum_{\ell=1}^dc_\ell P_{\ell,z_i \overline{z_j}}\left((1-\zeta)V,(1-\overline{\zeta})\overline{V}\right) = \left(1 - \overline \zeta\right)^{D_1-2}Q_{i\overline j}(\zeta) \\
\\
\bar{\zeta}^{k_0}\sum_{\ell=1}^dc_\ell P_{\ell,\overline{z_i} \overline{z_j}}\left((1-\zeta)V,(1-\overline{\zeta})\overline{V}\right) = \left(1 - \overline \zeta\right)^{D_1-2}S_{\bar i \bar j}(\zeta) \\
\end{array}\right.
\end{equation}
Here we write  $\frac{\partial^2 P_{\ell}}{\partial z_i \partial \overline{z_j}}$  (resp. $\frac{\partial^2 P_{\ell}}{\partial \overline z_i \partial \overline{z_j}}$) as $P_{\ell,z_i \overline{z_j}}$ (resp. $P_{\ell,\overline{z_i} \overline{z_j}}$). 
We claim that $Q_{i\overline j}$ are holomorphic polynomials of degree at most $2k_0 + 1$, and  is divisible by $\zeta^{D_1-1}$. For clarity, we prove this claim in the special case of $V=(1,\ldots,1)$ since an arbitrary choice of $V$ can only 
increase (resp. decrease) the degree of the  first nonvanishing term (resp. the degree).
Setting $U(\zeta):=\displaystyle \zeta^{k_0}\sum_{\ell=1}^d c_\ell P_{\ell,z_i \overline{z_j}}\left(1-\zeta,1-\overline{\zeta}\right)$, we have 
$$\
\begin{array}{lll} 
U(\zeta) &=& \displaystyle\zeta^{k_0}\sum_{\ell=1}^d c_\ell \sum_{\substack{|I| + |J| = D_\ell \\ D_\ell-k_\ell\leq |I| \leq k_\ell}}\widetilde{\alpha^\ell_{IJ}} (1-\zeta)^{|I|-1} \left(1-\overline{\zeta}\right)^{|J|-1}\\
\\
& = & \displaystyle \sum_{\ell=1}^d c_\ell \sum_{\substack{|I| + |J| = D_\ell \\ D_\ell-k_\ell\leq |I| \leq k_\ell}}\widetilde{\alpha^\ell_{IJ}}(-1)^{|I|-1}\zeta^{|I|+k_0-1} \left(1-\overline{\zeta}\right)^{D_\ell-2} \\
\\
& = & \displaystyle \left(1-\overline{\zeta}\right)^{D_1-2} \sum_{\ell=1}^d c_\ell \sum_{\substack{|I| + |J| = D_\ell \\ D_\ell-k_\ell\leq |I| \leq k_\ell}}\widetilde{\alpha^\ell_{IJ}}(-1)^{|I|-1}\zeta^{|I|+k_0-1-D_\ell+D_1} \left(1-\zeta\right)^{D_\ell-D_1}.\\
\end{array}
$$
The claim then follows from the fact that $|I|+k_0-1-D_{\ell}+D_1\geq D_1-1$
since $D_1\leq D_{\ell}$ for any $\ell$  and $k_0=\max\{k_1,\ldots,k_d\}$.  
Note also that $S_{\bar i \bar j}$ are antiholomorphic polynomials since, again taking $V=(1,\ldots,1)$, we have
$$
\begin{array}{lll} 
\displaystyle \bar{\zeta}^{k_0}\sum_{\ell=1}^dc_\ell P_{\ell,\overline{z_i} \overline{z_j}}\left(1-\zeta,\overline{1-\zeta}\right) &=& \displaystyle\bar{\zeta}^{k_0}\sum_{\ell=1}^d c_\ell \sum_{\substack{|I| + |J| = D_\ell \\ D_\ell-k_\ell\leq |I| \leq k_\ell}}\widetilde{\alpha^\ell_{IJ}} (1-\zeta)^{|I|} \left(1-\overline{\zeta}\right)^{|J|-2}\\
\\
& = & \displaystyle \sum_{\ell=1}^d c_\ell \sum_{\substack{|I| + |J| = D_\ell \\ D_\ell-k_\ell\leq |I| \leq k_\ell}}\widetilde{\alpha^\ell_{IJ}}(-1)^{|I|-1} \bar{\zeta}^{k_0-|I|} \left(1-\overline{\zeta}\right)^{D_\ell-2} \\
\\
& = & \displaystyle \left(1-\overline{\zeta}\right)^{D_1-2} \sum_{\ell=1}^d c_\ell \sum_{\substack{|I| + |J| = D_\ell \\ D_\ell-k_\ell\leq |I| \leq k_\ell}}\widetilde{\alpha^\ell_{IJ}}(-1)^{|I|-1}\bar{\zeta}^{k_0-|I|} \left(1-\overline{\zeta}\right)^{D_\ell-D_1}.\\
\end{array}
$$
Finally, similarly to \cite{be-de-la}, we define the following $n\times n$ matrix valued map  
\begin{equation}\label{eqQ}
Q(\zeta)= (Q_{i\overline j}(\zeta))
\end{equation}
for $\zeta \in \partial \Delta$. This map plays an important role in this paper and in the construction of stationary discs established in Theorem \ref{theodiscs}. In case $Q(\zeta)$ is invertible for all $\zeta \in \partial \Delta$, we denote by ${\rm ind}\left(-\overline{Q^{-1}}Q\right)$ 
the winding number at the origin of the function $\zeta \mapsto {\rm det}\left(-\overline{Q^{-1}(\zeta)}Q(\zeta)\right)$, that is,
$$\frac{1}{2\pi i} \int_{\partial \Delta} \frac{\left({\rm det}\left(-\overline{Q^{-1}(\zeta)}Q(\zeta)\right)\right)'}{{\rm det}\left(-\overline{Q^{-1}(\zeta)}Q(\zeta)\right)}{\rm d}\zeta.$$  We can now define:
 \begin{defi}\label{defiadmi}
 Let $M_H$ be a model given by \eqref{eqmod}. We say that the pair $(c,V) \in \R^d\times \C^n\setminus\{0\}$ is {\it admissible} if the matrix  $Q(\zeta)$ defined in \eqref{eqQ} is invertible for all $\zeta \in \partial \Delta$. The submanifold $M_H$ is said to be {\it admissible} if it admits an admissible  pair $(c,V)$.       
 \end{defi}
For real hypersurfaces in $\C^{n+1}$, this condition already appears in \cite{be-de-la} and is crucial in the construction of generalized stationary discs and their application in the  jet determination problem.  We note that this property is invariant under holomorphic changes of coordinates that preserve the canonical form \eqref{eqmod} since they are linear changes with the 
properties given by Proposition 9.2.6 in \cite{ber}. 
\begin{remark}\label{remnondeg} The two notions of admissibility and strong Levi nondegeneracy coincide in the case of a quadric model. Indeed, 
in such a case, we have $Q(\zeta)=\zeta \sum_{\ell=1}^dc_\ell A\ell$, where $A_\ell, \ell=1,\ldots,d$ is given by \eqref{eqA_j}. So, the map $Q(\zeta)$   is invertible for all $\zeta \in \partial\Delta$ 
if and only if $\sum_{j=1}^dcjAj$ is invertible. Note also that we have $S_{\overline{ij}}\equiv 0$ for all $i,j=1\ldots,n$. 
\end{remark}

The following example is due to Bertrand and Della Sala.
\begin{example}
The hypersurface 
$$\left\{\Re e w = |z|^4+2t\Re e (z^3\overline{z})\right\} \subset \C^2$$
is admissible whenever $0<t<2/3$ and is not admissible for $t\geq 2/3$.       
\end{example}

\begin{example} As already pointed out in the introduction, if $M_H \subset \C^{n+d}$ is of the form \eqref{eqmod} with $P_1(z,\overline{z})=\transp \overline{z} A_1 z$ where $A_1$ is Hermitian invertible, and {\it any} $P_2,\ldots,P_d$ (possibly zeros),
then $M_H$ is admissible; the pair $(c_1,\ldots,c_d)=(1,0,\ldots,0)$ and any $V \in \C^n\setminus\{0\}$ being admissible.     
\end{example}

\begin{example} The model submanifold $M_H \subset \C^4$ given by 
$$P_1(z,\overline{z})=|z_1|^{D_1}+|z_2|^{D_1} \ \ \mbox{ and }\ \ P_2(z,\overline{z})=|z_1|^{D_2}+|z_2|^{D_2}$$ 
is admissible. Indeed, we have 
$$Q_{1\overline 1}=Q_{2\overline 2}=c_1(D_1/2)^2(-1)^{D_1/2-1}\zeta^{k_0+D_1/2-1}+ c_2(D_2/2)^2(-1)^{D_2/2-1}(1-\overline{\zeta})^{D_2-D_1}\zeta^{k_0+D_2/2-1} $$ and 
$$Q(\zeta)= \left(\begin{matrix}
Q_{1\overline 1} (\zeta)&  0   \\
0  &   Q_{2\overline 2}(\zeta) \\
\end{matrix}\right)$$
For  $(c_1,c_2)=(1,0)$ and $V=(1,1)$ then $Q(\zeta)$  is invertible for all $\zeta \in \partial \Delta$. 
Note, however, that for  the pair $(c_1,c_2)=(0,1)$ and $V=(1,1)$ then  $Q(\zeta)$ is noninvertible at $\zeta=1$.
\end{example}

\begin{example} Consider a decoupled model $M_H$ of the form \eqref{eqmod} with 
$$P_1(z_1,\overline{z_1})=\sum_{j=D_1-k_1}^{k_1} \alpha_j z_1^j\overline{z_1}^{D_1-j}, \ \ P_2(z_2,\overline{z_2})=\sum_{j=D_2-k_2}^{k_2} \beta_j z_2^j\overline{z_2}^{D_2-j}.$$
In that case, unless $D_1=D_2$, no such model is admissible. Indeed, if $D_1<D_2$, then for any $c_1,c_2 \in \R$ and $V\in \C^2\setminus\{0\}$, the matrix $Q(\zeta)$ is equal to     
$$ \left(\begin{matrix}
c_1\sum_{j=D_1-k_1}^{k_1} \tilde{\alpha_j} \zeta^{k_0+j-1}V_1^{j-1}\overline{V_1}^{D_1-j-1} &  0   \\
0  &  c_2(1-\zeta)^{D_2-D_1}  \sum_{j=D_2-k_2}^{k_2} \tilde{\beta_j} \zeta^{k_0+j-1}V_2^{j-1}\overline{V_2}^{D_2-j-1} \\
\end{matrix}\right)$$
and is  singular at $\zeta=1$. Although the methods in the present paper do not apply to such manifolds, we emphasize that these are treated in  \cite{al-be-mc-ou-zo}. In case $D_1=D_2$, then  $M_H$ is admissible if and only if each of the two hypersurfaces  $\{\Re e w_1=P_1(z_1,\overline{z_1})\}\subset \C^2_{(z_1,w_1)}$ and 
$\{\Re e w_2=P_2(z_2,\overline{z_2})\}\subset \C^2_{(z_2,w_2)}$  is admissible.   

\end{example}

 \section {Generalized stationary discs and the Riemann-Hilbert Problem}\label {coucou2} 
\subsection{Function spaces}
In this section, we introduce the functional spaces we need in our context.  
For an integer $k \geq 0$ and a real number $0< \alpha<1$, we set $\mathcal C^{k,\alpha}=\mathcal C^{k,\alpha}(\partial\Delta,\R)$ to be the space of real-valued functions  defined on $\partial\Delta$ of class 
$\mathcal{C}^{k,\alpha}$. This space is equipped with the  norm 
$$\|f\|_{\mathcal{C}^{k,\alpha}}=\sum_{j=0}^{k}\|f^{(j)}\|_\infty+
\underset{\zeta\not=\eta\in \partial\Delta}{\mathrm{sup}}\frac{\|f^{(k)}(\zeta)-f^{(k)}(\eta)\|}{|\zeta-\eta|^\alpha},$$
where $\|f^{(j)}\|_\infty=\underset{\partial\Delta}{\mathrm{max}}\|f^{(j)}\|$.
We write  $\mathcal C_\C^{k,\alpha} = \mathcal C^{k,\alpha} + i\mathcal C^{k,\alpha}$ and equip this space with the norm 
\begin{equation}\label{eqnorm1}
\|f\|_{\mathcal{C}_{\C}^{k,\alpha}}=
\|\Re e  f\|_{\mathcal{C}^{k,\alpha}}+\|\Im m f\|_{\mathcal{C}^{k,\alpha}}.
\end{equation}
The subspace of {\it analytic discs} in $\mathcal C_{\C}^{k,\alpha}$ is denoted by $\mathcal A^{k,\alpha}$  and consists of functions defined on $\overline{\Delta}$, holomorphic on $\Delta$ and such that their restriction on 
$\partial\Delta$ is in  $\mathcal C_\C^{k,\alpha}$.

\vspace{0.5cm}
In this paper, we  work with analytic discs that enjoy pointwise constraints in the sense that they vanish at $\zeta=1$  with a given order. For this reason, we introduce adapted functional spaces. Let $m\geq 0$ be an integer. We define $\mathcal A^{k,\alpha}_{0^m}$
as the subspace of 
$\mathcal C_{\C}^{k,\alpha}$ of functions that can be written as $(1-\zeta)^m f$, with  $f\in \mathcal A^{k,\alpha}$, endowed with the  norm 
\begin{equation}\label{eqnorm2}
\|(1-\zeta)^m f\|_{\mathcal A^{k,\alpha}_{0^m}}
=\vnorm{ f }_{\mathcal{C}_{\C}^{k,\alpha}},
\end{equation}
Finally, we denote by $\mathcal C_{0^m}^{k,\alpha}$ the subspace of $\mathcal C^{k,\alpha}$ of functions which can be written as $(1-\zeta)^m v$ with $v\in \mathcal C_\C^{k,\alpha}$, and equip this space with the norm
$$\|(1-\zeta)^m f\|_{\mathcal C_{0^m}^{k,\alpha}}=\vnorm{ f }_{\mathcal C_\C^{k,\alpha}}.$$
We point out that both $\mathcal A^{k,\alpha}_{0^m}$ and $\mathcal C_{0^m}^{k,\alpha}$ are Banach spaces. Finally, when $m=1$, we simply write $\mathcal A^{k,\alpha}_{0}$ and $\mathcal C_{0}^{k,\alpha}$.

 \subsection{Generalized stationary discs}

 Let $M=\{r=0\}\subset \C^{n+d}$ be a smooth  real submanifold of the form (\ref{eqdeg}). An analytic disc $f$ 
is {\it attached to  $M$} when $f(\partial \Delta) \subset M$. Following \cite{be-de1, be-de-la}, we define a higher order notion of stationary discs (see Lempert \cite{le} and Tumanov \cite{tu} for the standard definition):
\begin{defi}
Let $k_0>0$  be a positive integer. A holomorphic disc $f: \Delta \to \C^{n+d}$ continuous up to  $\partial \Delta$ and attached to  $M=\{r=0\}$ is a {\it $k_0$-stationary disc for $M$} if there 
exists a  holomorphic lift $\bm{f}=(f,\tilde{f})$ of $f$ to the cotangent bundle $T^*\C^{n+d}$, continuous up to 
 $\partial \Delta$ and such that, for all $\zeta \in \partial\Delta$, $\bm{f}(\zeta) \in \mathcal{N}^{k_0}M(\zeta)$ where
\begin{equation}\label{eqcon}
\mathcal{N}^{k_0}M(\zeta):= \left\{(z,w,\tilde{z},\tilde{w}) \in T^*\C^{n+d} \ | \ (z,w) \in M, (\tilde{z},\tilde{w}) \in 
\zeta^{k_0} N^*_{(z,w)} M\setminus \{0\}\right\},
\end{equation}
 where 
$$N^*_{(z,w)} M=\spanc_{\R}\{\partial r_1(z,w), \ldots, \partial r_d(z,w)\}$$ is the conormal fiber at $(z,w)$ of the submanifold $M$. Here 
$\partial r_\ell=\left(\partial_z r_\ell, \partial_w r_\ell\right)$ for $\ell=1,\ldots,d. $  
The map $\bm{f}=(f,\tilde{f})$ is called a {\it $k_0$-stationary lift for $M$} and we denote by $\mathcal{S}^{k_0}(M)$ the set of these lifts  with $f$ nonconstant. 
\end{defi}
When we do not specify $k_0$, we refer to these discs as {\it generalized stationary discs}. 
We note that an analytic disc $f$ is $k_0$-stationary for $M$ if and only if there exist $d$ real valued functions $c_1, \ldots, c_d: \partial \Delta \to \R$ with $\sum_{\ell=1}^dc_\ell(\zeta)\partial r_\ell(0)\neq 0$ for all $\zeta \in \partial \Delta$  and such that the map 
$$\zeta \mapsto \zeta^{k_0} \sum_{\ell=1}^dc_\ell(\zeta)\partial r_\ell\left(f(\zeta), \overline{f(\zeta)}\right)$$
 defined on $\partial \Delta$ extends holomorphically on $\Delta$. We also point out  that the set of  small $k_0$-stationary discs is invariant under CR automorphisms (see \cite{be-bl-me} or Proposition 2.3 \cite{be-de-la}).  
In the next example, we show that one can always construct a basic family of $k_0$-stationary discs for {\it any} $M_H$ \eqref{eqmod}.   
\begin{example}
Consider a model submanifold $M_H= \{\rho=0\}\subset \C^{n+d} $ of the form (\ref{eqmod}). 
We first note that 
\begin{equation*}
\begin{cases}
\displaystyle \partial \rho_1 = (\partial_z \rho_1, \partial_w \rho_1)=  \left(-P_{1,z_1}(z,\overline{z}),\ldots,-P_{1,z_n}(z,\overline{z}), \frac{1}{2},0,\ldots,0\right)\\
\hspace{4cm} \vdots\\
\displaystyle \partial \rho_d = (\partial_z \rho_d, \partial_w \rho_d)= \left(-P_{d,z_1}(z,\overline{z}),\ldots,-P_{d,z_n}(z,\overline{z}), 0,\ldots,0,\frac{1}{2}\right)
 \end{cases}
\end{equation*}
Set, as in Subsection 1.1, $k_0:=\max\{k_1,\ldots,k_d\}$. It is easy to show that the disc 
\begin{equation}\label{eqdisini}
\bm{f_0}(\zeta)=(\underbrace{h_0(\zeta),g_0(\zeta)}_{f_0(\zeta)},\tilde{h_0}(\zeta),\tilde{g_0}(\zeta))=\left((1-\zeta)V,g_0(\zeta),\tilde{h_0}(\zeta),\frac{c_1}{2}\zeta^{k_0},\ldots,\frac{c_d}{2}\zeta^{k_0}\right)\\
\end{equation}
where $(c_1,\ldots, c_d) \in \R^d$ and $V \in \C^n\setminus\{0\}$, is a $k_0$-stationary lift for  $M_H$. The component $g_0$ is classically determined  by (\ref{eqmod})  by means of the Hilbert transform and 
 $$\tilde{h_0}(\zeta)=\zeta^{k_0}\left(\sum_{\ell=1}^dc_\ell P_{\ell,z_1}\left((1-\zeta)V,(1-\overline{\zeta})\overline{V}\right),\ldots,\sum_{\ell=1}^dc_\ell P_{\ell,z_n}\left((1-\zeta)V,(1-\overline{\zeta})\overline{V}\right)\right).$$
\end{example}

Although this example applies to nonadmissible models, it becomes much more relevant when $M_H$ is admissible. Indeed,
our goal is to construct generalized stationary lifts by deforming both the model  $M_H$ and the initial lift $\bm{f_0}$ 
\eqref{eqdisini}. In this context, it is important to consider an admissible model $M_H$ with an initial lift $\bm{f_0}$ given by \eqref{eqdisini} where  the parameters $(c_1,\ldots,c_d) \in \R^d$ and $V\in \C^n\setminus\{0\}$  form an admissible pair for $M_H$. We recall that, as opposed to the strongly Levi nondegenerate setting, the conormal bundle of a
degenerate model $M_H$ is not totally real (see \cite{we, tu}) and thus, classical techniques to construct stationary lifts 
(see e.g. \cite{be-bl, be-bl-me}) do not apply any longer. The admissibility condition is crucial since it allows the reduction to a setting where the 
initial lift $\bm{f_0}$ is, up to a factorization by powers of $(1-\zeta)$, attached to a totally real submanifold of the cotangent bundle, and therefore can  be 
deformed to obtain nearby lifts.

 \subsection{A  Riemann-Hilbert type problem}\label{subrie}
In this subsection, we show how to describe the fibration $\mathcal{N}^{k_0}M(\zeta)$ for a submanifold $M\subset \C^{n+d}$ of the form \eqref{eqdeg}, with  $2n+2d$ real  equations. Consider first a model submanifold $M_H= \{\rho=0\}$ given by (\ref{eqmod}). In that case, we obtain explicit defining equations for the corresponding fibration. Indeed, we first point out that 
\begin{eqnarray*}
(z,w,\tilde{z},\tilde{w}) \in \mathcal{N}^{k_0}M_H(\zeta) &\Leftrightarrow&  
\left\{
\begin{array}{lll} 
\rho_1(z,w,\overline{z},\overline w)=...=\rho_d(z,w,\overline{z},\overline w)=0 \\
\\
\displaystyle \exists \ c_{\ell}: \partial \Delta\rightarrow \R, 
 (\tilde{z},\tilde{w})=\zeta^{k_0} \sum_{\ell=1}^dc_{\ell}(\zeta)\partial \rho_{\ell}\left(f(\zeta), \overline{f(\zeta)}\right)
\end{array}
\right.
\end{eqnarray*}
Due to the form of $\rho$, we have 
$$ \sum_{\ell=1}^dc_{\ell}(\zeta)\partial \rho_{\ell}(f(\zeta), \overline{f(\zeta)})=
\left(-\sum_{\ell=1}^dc_{\ell}(\zeta)P_{\ell,z_1}(z,\overline{z}),\ldots,-\sum_{\ell=1}^dc_{\ell}(\zeta)P_{\ell,z_n}(z,\overline{z}),\frac{c_1(\zeta)}{2},\ldots,\frac{c_d(\zeta)}{2}\right)$$
and so
\begin{eqnarray*}
(z,w,\tilde{z},\tilde{w}) \in \mathcal{N}^{k_0}M_H(\zeta) &\Leftrightarrow&  
\left\{
\begin{array}{lll} 
\rho_1(z,w,\overline{z},\overline w)=...=\rho_d(z,w,\overline{z},\overline w)=0 \\
\\
\tilde{z_j}=-2\sum_{\ell=1}^d\tilde{w}_{\ell}(\zeta)P_{\ell,z_j}(z,\overline{z}), \ j=1\ldots,n\\
 \\
 \zeta^{-k_0}\tilde{w_1},\ldots, \zeta^{-k_0}\tilde{w_d} \in \R.
 \end{array}
\right.
\end{eqnarray*}
It then follows that the $2n+2d$ defining equations of $\mathcal{N}^{k_0}M_H(\zeta)$ are given by
\begin{equation*}\small
\left\{
\begin{array}{lll} 

\tilde{\rho}_1(\zeta)(z,w,\tilde{z},\tilde{w}) & = &  \Re e w_1 - P_1(z,\overline z)=0\\
& \vdots & 
\\	
\tilde{\rho}_d(\zeta)(z,w,\tilde{z},\tilde{w}) & = & \Re e w_d - P_d(z,\overline z)=0\\
\\
\displaystyle \tilde{\rho}_{d+1}(\zeta)(z,w,\tilde{z},\tilde{w}) & = & \left(\tilde{z}_1+2\sum_{\ell=1}^d\tilde{w}_{\ell}P_{\ell,z_1}(z,\overline{z})\right) + 
\left(\overline{\tilde{z}_1+2\sum_{\ell=1}^d\tilde{w}_{\ell}P_{\ell,z_1}(z,\overline{z})}\right) = 0\\
\\

\tilde{\rho}_{d+2}(\zeta)(z,w,\tilde{z},\tilde{w}) & = & i\left(\tilde{z}_1+2\sum_{\ell=1}^d\tilde{w}_{\ell}P_{\ell,z_1}(z,\overline{z})\right) - 
i\left(\overline{\tilde{z}_1+2\sum_{\ell=1}^d\tilde{w}_{\ell}P_{\ell,z_1}(z,\overline{z})}\right) = 0\\ 
&\vdots&\\
\tilde{\rho}_{2n+d-1}(\zeta)(z,w,\tilde{z},\tilde{w}) & = & \left(\tilde{z}_n+2\sum_{\ell=1}^d\tilde{w}_{\ell}P_{\ell,z_n}(z,\overline{z})\right) + 
\left(\overline{\tilde{z}_n+2\sum_{\ell=1}^d\tilde{w}_{\ell}P_{\ell,z_n}(z,\overline{z})}\right) = 0\\

\\
\tilde{\rho}_{2n+d}(\zeta)(z,w,\tilde{z},\tilde{w}) & = & i\left(\tilde{z}_n+2\sum_{\ell=1}^d\tilde{w}_{\ell}P_{\ell,z_n}(z,\overline{z})\right) - 
i\left(\overline{\tilde{z}_n+2\sum_{\ell=1}^d\tilde{w}_{\ell}P_{\ell,z_n}(z,\overline{z})}\right) = 0\\
\\
\tilde{\rho}_{2n+d+1}(\zeta)(z,w,\tilde{z},\tilde{w}) & = & \displaystyle i\zeta^{-k_0}\tilde{w}_1-i\zeta^{k_0}\overline{\tilde{w}_1} = 0.\\
&\vdots& 
\\\tilde{\rho}_{2n+2d}(\zeta)(z,w,\tilde{z},\tilde{w}) & = &\displaystyle  i\zeta^{-k_0}\tilde{w}_d-i\zeta^{k_0}\overline{\tilde{w}_d} = 0.\\

\end{array}
\right.
\end{equation*}
We set $\tilde{\rho}:=(\tilde{\rho}_1,\cdots,\tilde{\rho}_{2n+2d})$. For a general submanifold $M=\{r=0\}$ of the form \eqref{eqdeg},  we denote by $\tilde{r}$ the corresponding defining functions of the fibration $\mathcal{N}^{k_0}M(\zeta)$. 

The idea of describing the fibration $\mathcal{N}^{k_0}M(\zeta)$ using real defining functions is convenient since it allows us to consider $k_0$-stationary lifts as solutions of a nonlinear 
Riemann-Hilbert type problem. More precisely, an analytic disc $\bm{f}:\Delta \mapsto T^*\C^{n+d}$ is a $k_0$-stationary 
lift for $M$ if and only if it satisfies the following Riemann-Hilbert problem
\begin{equation}\label{eqstatrh}
\tilde{r}(\zeta)(\bm{f}(\zeta))=0 \ \mbox{ for all } \zeta \in \partial \Delta.
\end{equation}
The study of this problem and the structure of its solutions relies on the values of special integers called the {\it partial indices} and the {\it Maslov index}. We briefly recall these concepts.
Let $G: \partial\Delta \to Gl_N(\C)$ be a smooth map, where  $Gl_N(\C)$ denotes the general linear group on $\C^N$. 
We consider a Birkhoff factorization 
(see Globevnik \cite{gl1}  or  Vekua \cite{ve}) of $-\overline{G^{-1}}G$ on $\partial \Delta$:
$$ -\overline{G(\zeta)^{-1}}G(\zeta)=
B^+(\zeta)
\begin{pmatrix}
	\zeta^{\kappa_1}& & & (0) \\ &\zeta^{\kappa_2} & & \\ & & \ddots & \\ (0)& & &\zeta^{\kappa_{N}}
\end{pmatrix}
B^-(\zeta),$$
where  $\zeta \in \partial \Delta$, $B^+: \bar{\Delta}\to Gl_N(\C)$ and 
$B^-: (\C \cup \infty)\setminus\Delta\to Gl_N(\C)$  
 are  smooth  maps, holomorphic on $\Delta$ and $\C \setminus \overline{\Delta}$ respectively. 
The integers $\kappa_1, \dots, \kappa_N$  are  the {\it partial indices} of 
$-\overline{G^{-1}}G$ and the {\it Maslov index} of $-\overline{G^{-1}}G$ is  their sum 
$\kappa:=\sum_{j=1}^N\kappa_j$. It is important to note (see \cite{gl2}) that
  $\kappa$ is also equal to the winding number at the origin of  the map
$$\zeta \mapsto \det\left(-\overline{G(\zeta)^{-1}}G(\zeta)\right).$$

 \section{Construction of generalized stationary discs and its consequences}\label{coucou3}

\subsection{The main result}
The construction of generalized stationary discs relies on the implicit function theorem with well adapted Banach spaces, 
which we discuss now. Let  $k\geq 0$  be an integer 
and let $0<\alpha<1$. We also consider a model  $M_H=\{\rho=0\}\subset \C^{n+d} $ of the form (\ref{eqmod}), and recall that $D_1\leq D_{\ell} \leq D_d$, $\ell=1,\ldots,d$  and $k_0=\max\{k_1,\ldots,k_d\}$.

We introduce the following subspace of analytic discs with pointwise constraints
\begin{equation}\label{eqdefY2}
Y:= \left(\mathcal A^{k,\alpha}_{0}\right)^{n+d}  \times \left(\mathcal A^{k,\alpha}_{0^{D_1-1}} \right)^n
\times \left(\mathcal A^{k,\alpha}\right)^d
\end{equation}
endowed with the product norm $\|\cdot\|$ associated with  (\ref{eqnorm1}) and (\ref{eqnorm2}). We  denote, for any real submanifold $M \subset \C^{n+d}$, the set of $k_0$-stationary lifts with pointwise constraints by
$$ \mathcal{S}^{k_0}_0(M)= \mathcal{S}^{k_0}(M) \cap Y.$$

We also define the space $X$ of defining functions we work with. The definition of this space  occurs already in previous works (see \cite{al-be-mc-ou-zo, be-de1,be-de-la}) and is strongly related to the norm introduced on $Y$. We refer to the below definition of the map 
$F$ \eqref{eqF1}   for more insight. 
Furthermore, we note that in the Levi nondegenerate case, $X$ coincides with all possible perturbations of the model $M_H$.    
Choose $\delta > 0$ large enough to ensure that $f_0\left(\overline \Delta\right)$, where $f_0$ is the disc defined in (\ref{eqdisini}), 
is contained in the polydisc ${\delta \Delta}^{n+d}\subset \C^{n+d}$.
We consider the affine Banach space $X$ of functions $r\in  \mathcal{C}^{k+3}\left(\delta\Delta
^{n+d}\right)$, where $k$ is the same as in the definition of $Y$, which can be written as
\begin{equation*}
r(z,w) = \rho(z,w) +\theta(z,\Im m w )
\end{equation*}
with $\theta=(\theta_1,\ldots,\theta_d)$ of the form 
\begin{equation}\label{eqallow}
\theta_\ell(z,\Im m w )= \sum_{|I|+ |J|=D_\ell+1} z^I \overline z^J  r_{\ell,IJ0}(z)+
\sum_{|S|=1}^{D_\ell}\sum_{|I|+ |J|=D_\ell-|S|} z^I \overline z^J (\Im m w)^{|S|}  r_{\ell,IJS}(z,\Im m w)
\end{equation}
where $r_{\ell,IJ0} \in  \mathcal C^{k+3}_\C\left(\delta\Delta^{n}\right)$ and 
$ r_{\ell,IJS} \in \mathcal  C^{k+3}_\C\left(\delta\Delta^n\times(-\delta,\delta)^d\right)$. We equip $X$ with the following norm 
\begin{equation}\label{eqnorm3}
\vnorm{ r }_{X} = \sup \vnorm{ r_{\ell,IJS} }_{\mathcal C^{k+3}}.
\end{equation}

We are now in a position to state and prove our main result. 
\begin{theo}\label{theodiscs}
Let  $M_H=\{\rho=0\} \subset \C^{n+d}$ be a model submanifold  of the form (\ref{eqmod}). 
Consider an initial stationary lift  $\bm{f_0}=\left(h_0,g_0,\tilde{h_0},\tilde{g_0}\right) \in \mathcal{S}^{k_0}_0(M_H)$ of the form \eqref{eqdisini} where $(c_1,\ldots,c_d) \in \R^d$ and $V\in \C^n\setminus\{0\}$ form an admissible pair for $M_H$. Then there exist open 
neighborhoods $U$ of $\rho$ in $X$  and $V$ of $0$  in $\R^m$, where $m={\rm ind}\left(-\overline{Q^{-1}}Q\right)+d(2k_0+1)-2n(D_1-2)$, a real number 
$\varepsilon>0$, and a map of class $\mathcal{C}^1$
$$\mathcal{F}:U \times V \to  Y $$
such that:
\begin{enumerate}[i.]
\item $\mathcal{F}(\rho,0)=\bm{f_0}$,
\item for all $r\in U$, the map 
$$\mathcal{F}(r,\cdot):V\to \left\{\bm{f} \in \mathcal{S}^{k_0}_0(\{r=0\})\ \ | \  \|\bm{f}-\bm{f_0}\|<\varepsilon\right\}$$
is one-to-one and onto.
\end{enumerate}
In particular, for any defining function  $r \in U$, the set $$\left\{\bm{f} \in \mathcal{S}^{k_0}_0(\{r=0\})\ \ | \  
\|\bm{f}-\bm{f_0}\|<\varepsilon\right\}$$
is a $\mathcal{C}^1$ real submanifold of dimension ${\rm ind}\left(-\overline{Q^{-1}}Q\right)+d(2k_0+1)-2n(D_1-2)$ of the Banach space of analytic discs $Y$.    
\end{theo}

\begin{proof}[Proof of Theorem \ref{theodiscs}]
In a neighborhood of $(\rho,\bm{f_0})$ in 
$X \times Y$,
 we define the following map 
 \begin{equation}\label{eqF1}
 F: X \times  Y
\to  \left(\mathcal C_0^{k,\alpha}\right)^d  \times \left(\mathcal C_{0^{D_1-1}}^{k,\alpha}\right)^{2n} \times   \left(\mathcal C^{k,\alpha}\right)^d
\end{equation}
by 
\begin{equation*}
F(r,\bm{f}):=\tilde{r}(\bm{f})
\end{equation*}
in the sense of \eqref{eqstatrh}. 
We emphasize that the norms defined on the spaces $Y$ and  on the target space of $F$ are natural and well 
adapted to the study of analytic discs with pointwise constraints (see e.g. \cite{be-de2}). Consequently, the space $X$  and its 
corresponding norm \eqref{eqnorm3} have been constructed to ensure that the map $F$ is well defined, in the sense that $\tilde{r}(\bm{f})$ belongs to the target space in \eqref{eqF1}, and is of class $\mathcal{C}^1$
(see Lemma 5.1 \cite{hi-ta}, Lemma 6.1 and Lemma 11.2 \cite{gl1}, Lemma 3.3 \cite{be-de1}). 
For any fixed defining function $r$, 
 the zero set of $F(r,\cdot)$ coincides with
$\mathcal{S}^{k_0}_0(\{r=0\})$. Then the proof of Theorem \ref{theodiscs} relies on the implicit function theorem applied to the map $F$ (see e.g. p.39 \cite{gl2}). 
Let $\bm{f}\mapsto \partial_2 F(\rho,\bm{f_0})\bm{f}$ be the partial derivative of $F$ with respect to the Banach space 
$Y$ at 
$(\rho,\bm{f_0})$. We have 
\begin{equation*}
\partial_2 F(\rho,\bm{f_0})\bm{f}=2\Re e  \left[\overline{G(\zeta)}\bm{f}\right]
\end{equation*}
where $G(\zeta)$ is the following square matrix of size $2n+2d$  
$$G(\zeta):=\left({\tilde\rho}_{\overline{z}}(\bm{f_0}),
{\tilde\rho}_{\overline{w}}(\bm{f_0}),
{\tilde\rho}_{\overline{\tilde{z}}}(\bm{f_0}),{\tilde\rho}_{\overline{\tilde{w}}}(\bm{f_0})\right).$$
To prove Theorem \ref{theodiscs}, we need to show the following two properties of $\partial_2 F(\rho,\bm{f_0})$: 
\begin{enumerate}[i.]
\item The map $\partial_2 F(\rho,\bm{f_0}): Y
\to  \left(\mathcal C_0^{k,\alpha}\right)^d  \times \left(\mathcal C_{0^{D_1-1}}^{k,\alpha}\right)^{2n} \times   \left(\mathcal C^{k,\alpha}\right)^d$ is onto.
\item The kernel of $\partial_2 F(\rho,\bm{f_0})$ is of real dimension ${\rm ind}\left(-\overline{Q^{-1}}Q\right)+d(2k_0+1)-2n(D_1-2)$.
\end{enumerate}

\vspace{0.5cm}

We first focus on i.
It is more convenient to reorder the given coordinates as $(w,z,\tilde{z},\tilde{w})$. In this context, discs $\bm{f}$ are of the form $(g,h,\tilde{h},\tilde{g})$. We still denote by $G(\zeta)$ the corresponding reordered matrix, that is, 
    $$G(\zeta):=\left({\tilde\rho}_{\overline{w}}(\bm{f_0}),
{\tilde\rho}_{\overline{z}}(\bm{f_0}),
{\tilde\rho}_{\overline{\tilde{z}}}(\bm{f_0}),{\tilde\rho}_{\overline{\tilde{w}}}(\bm{f_0})\right).$$
It turns out that this matrix is block upper triangular 
\begin{equation}\label{eqG}
G(\zeta)=\left(\begin{array}{cccccccccccc}
\frac{1}{2}I_d &  & (*)\\
   & G_2(\zeta) &  \\
 (0) &   &  -i\zeta^{k_0} I_d\\
\end{array}\right),
\end{equation}
where $I_d$ is the $d\times d$ identity matrix and $G_2(\zeta)$ is the square matrix of size $2n$ given by 
\begin{equation}\label{eqG_2}
\Small
\left(\begin{matrix}
\sum_{\ell=1}^d(\zeta^{k_0}c_\ell P_{\ell,z_1\overline{z_1}}+\overline{\zeta}^{k_0}c_\ell P_{\ell,\overline{z_1}\overline{z_1}}) & \ldots &  \sum_{\ell=1}^d(\zeta^{k_0}c_\ell P_{\ell,z_1\overline{z_n}}+\overline{\zeta}^{k_0}c_\ell P_{\ell,\overline{z_1}\overline{z_n}}) & 1 & \ldots & 0 \\
i\sum_{\ell=1}^d(\zeta^{k_0}c_\ell P_{\ell,z_1\overline{z_1}}-\overline{\zeta}^{k_0}c_\ell P_{\ell,\overline{z_1}\overline{z_1}}) & \ldots & i\sum_{\ell=1}^d(\zeta^{k_0}c_\ell P_{\ell, z_1\overline{z_n}}-\overline{\zeta}^{k_0}c_\ell P_{\ell,\overline{z_1}\overline{z_n}})    & -i & \ldots & 0 \\
\vdots & \vdots & \vdots &\vdots  & \ddots & \vdots \\
\sum_{\ell=1}^d(\zeta^{k_0}c_\ell P_{\ell,\overline{z_1}z_n}+\overline{\zeta}^{k_0}c_\ell P_{\ell,\overline{z_1}\overline{z_n}}) & \ldots & \sum_{\ell=1}^d(\zeta^{k_0}c_\ell P_{\ell,z_2\overline{z_n}}+\overline{\zeta}^{k_0}c_\ell P_{\ell,\overline{z_2}\overline{z_n}})    & 0 &\ldots &  1 \\
i\sum_{\ell=1}^d(\zeta^{k_0}c_\ell P_{\ell,\overline{z_1}z_n}-\overline{\zeta}^{k_0}c_\ell P_{\ell,\overline{z_1}\overline{z_n}}) & \ldots & i\sum_{\ell=1}^d(\zeta^{k_0}c_\ell P_{\ell,z_2\overline{z_n}}-\overline{\zeta}^{k_0}c_\ell P_{\ell,\overline{z_2}\overline{z_n}})    & 0 &\ldots &  -i \\
\end{matrix}\right).
\end{equation}
Due to the form of $G$,  the surjectivity of $\partial_2 F(\rho,\bm{f_0})$ follows from the surjectivity of the linear map 
$$L_2:\left(\mathcal A^{k,\alpha}_{0}\right)^n  \times \left(\mathcal A^{k,\alpha}_{0^{D_1-1}} \right)^n \to \left(\mathcal C_{0^{D_1-1}}^{k,\alpha}\right)^{2n}$$
defined by $$L_2((1-\zeta)h,(1-\zeta)^{D_1-1}\tilde{h})=2\Re e  \left[\overline{G_2(\zeta)}((1-\zeta)h,(1-\zeta)^{D_1-1}\tilde{h})\right]$$
After column permutations, we have (still denoting  by $G_2$ the corresponding matrix)
\begin{equation*}
\Small
G_2(\zeta)=\left(\begin{matrix}
1 & \sum_{\ell=1}^d(\zeta^{k_0}c_\ell P_{\ell,z_1\overline{z_1}}+\overline{\zeta}^{k_0}c_\ell P_{\ell,\overline{z_1}\overline{z_1}}) & \ldots & 0 & \sum_{\ell=1}^d(\zeta^{k_0}c_\ell P_{lz_1\overline{z_n}}+\overline{\zeta}^{k_0}c_\ell P_{\ell,\overline{z_1}\overline{z_n}})  \\
-i & i\sum_{\ell=1}^d(\zeta^{k_0}c_\ell P_{\ell,z_1\overline{z_1}}-\overline{\zeta}^{k_0}c_\ell P_{\ell,\overline{z_1}\overline{z_1}}) & \ldots &0    &  i\sum_{\ell=1}^d(\zeta^{k_0}c_\ell P_{\ell, z_1\overline{z_n}}-\overline{\zeta}^{k_0}c_\ell P_{\ell,\overline{z_1}\overline{z_n}})  \\
\vdots &\vdots & \vdots &\vdots& \vdots \\
0& \sum_{\ell=1}^d(\zeta^{k_0}c_\ell P_{\ell,\overline{z_1}z_n}+\overline{\zeta}^{k_0}c_\ell P_{\ell,\overline{z_1}\overline{z_n}}) & \ldots & 1   & \sum_{\ell=1}^d(\zeta^{k_0}c_\ell P_{\ell,z_n\overline{z_n}}+\overline{\zeta}^{k_0}c_\ell P_{\ell,\overline{z_n}\overline{z_n}})   \\
0& i\sum_{\ell=1}^d(\zeta^{k_0}c_\ell P_{\ell,\overline{z_1}z_n}-\overline{\zeta}^{k_0}c_\ell P_{\ell,\overline{z_1}\overline{z_n}}) & \ldots & -i   & i\sum_{\ell=1}^d(\zeta^{k_0}c_\ell P_{\ell,z_n\overline{z_n}}-\overline{\zeta}^{k_0}c_\ell P_{\ell,\overline{z_n}\overline{z_n}})  \\
\end{matrix}\right)
\end{equation*}
which we write, using \eqref{eqQS}, as 
\begin{equation}\label{eqG_2fact}
\underbrace{\left(\begin{matrix}
 1 & Q_{1\bar{1}}+  S_{\bar{1}\bar{1}}& \ldots & 0 & Q_{1\bar{n}}+  S_{\bar{1}\bar{n}}  \\
   -i&    iQ_{1\bar{1}}-i  S_{\bar{1}\bar{1}} & \ldots & 0 & iQ_{1\bar{n}}-i  S_{\bar{1}\bar{n}}  \\
  \vdots  & \vdots &\vdots &   \vdots & \vdots \\
    0&   Q_{n\bar{1}}+  S_{\bar{n}\bar{1}} & \ldots &   1 & Q_{n\bar{n}}+  S_{\bar{n}\bar{n}} \\
     0&  iQ_{n\bar{1}}-  iS_{\bar{n}\bar{1}}  & \ldots & -i & 
      iQ_{n\bar{n}}-i  S_{\bar{n}\bar{n}}
    \end{matrix}\right)}_{\widetilde{G_2}(\zeta)}  \underbrace{\left(\begin{matrix}
 1&  &  & & (0) \\
       &  (1-\overline{\zeta})^{D_1-2}  &&  &  \\
       &   & \ddots   & &  \\
       &  & &1  &  \\
     (0) &  &    &  & (1-\overline{\zeta})^{D_1-2} 
    \end{matrix}\right)}_{D(\zeta)} 
\end{equation}
Since by assumption $(c_1,\ldots,c_d)$ and $V$ is an admissible pair and since $\det \widetilde{G_2}=(2i)^n\rm{det}(Q)$, 
 the matrix $\widetilde{G_2}(\zeta)$ is invertible for all $\zeta \in \partial \Delta$. The linear operator
\begin{equation*}
\widetilde{L_2}: \left(\mathcal A^{k,\alpha}_{0^{D_1-1}} \right)^{2n} \to \left(\mathcal C_{0^{D_1-1}}^{k,\alpha}\right)^{2n}
\end{equation*}
defined by $$L_2=\widetilde{L_2}\circ \overline{D}$$ is of the form considered in Theorem 2.1 \cite{be-de2} and its surjectivity is equivalent to the one of   $L_2$. 
We then need to study the matrix $ \overline{\widetilde{G_2}^{-1}}\widetilde{G_2}$. We first note that the 
matrix $\widetilde{G_2}(\zeta)$ is exactly of the form of the matrix $A$ used in Lemma 4.4 \cite{be-de-la} with $s_1=\ldots=s_n=0$. We then write
\[\overline {\widetilde{G_2}^{-1}(\zeta)} \widetilde{G_2}(\zeta) =: \frac{1}{\overline {\det \widetilde{G_2}}} A'(\zeta) = \frac{1}{\overline{(2i)^n\rm{det}(Q(\zeta))}} A'(\zeta).\]
Following the computations performed in Lemma 4.4 \cite{be-de-la}, and carrying the same notation, we obtain the following expression
\begin{equation*}
 A'(\zeta)=-(2i)^n \begin{pmatrix}

C_{2,1;1} &  a_{1,2} &  C_{3,1;1}  &  a_{1,4}&  \cdots & C_{n+1,1;1}   &  a_{1,2n} \\

c_{1,1;1}&   C_{1,2}^1  &  c_{2,1;1}  &  C_{1,2}^2  &  \cdots & c_{n,1;1} & C_{1,2}^n   \\
 
C_{2,1;2} &  a_{3,2} &  C_{3,1;2} &  a_{3,4} &  \cdots & C_{n+1,1,2}  &  a_{3,2n}\\

-c_{1,1;2} &  C_{1,3}^1   &-c_{2,1;2}    &  C_{1,3}^2  &  \cdots & -c_{n,1;2} & C_{1,3}^n  \\

\vdots  &  \vdots&   \vdots &  \vdots&  \cdots & \vdots &  \vdots \\

C_{2,1;n}&  a_{2n-1,2} &   C_{3,1;n}  &  a_{2n-1,4} &  \cdots & C_{n+1,1;n} &  a_{2n-1,2n} \\
 
 \frac{ c_{1,1;n}}{(-1)^{n+1}} &  C_{1,n+1}^1   &      \frac{c_{2,1;n}}{(-1)^{n+1}}  &  C_{1,n+1}^2  &  \cdots &     \frac{c_{n,1;n}}{(-1)^{n+1}} & C_{1, n+1}^n  \\

\end{pmatrix}.
\end{equation*}
If we  denote by $C_p$ the $p^{\rm th}$ column of $A'(\zeta)$, then performing the column operation 
\begin{equation}\label{eqop}
C_{2p} \to C_{2p}-\sum_{j=1}^{n} {S}_{\overline{j p}}C_{2j-1}
\end{equation} 
for each $p=1,\cdots,n$,
transforms $A'(\zeta)$ into
\begin{equation*}
 -(2i)^n\left(\begin{matrix}

C_{2,1;1} &  Q_{1\overline 1} \overline{\rm{det}(Q(\zeta))}&   C_{3,1;1}  & Q_{1\overline 2} \overline{\rm{det}(Q(\zeta))} &  \cdots & C_{n+1,1;1}  &  Q_{1\overline n} \overline{\rm{det}(Q(\zeta))} \\

c_{1,1;1} &   0  &  c_{2,1;1}  &  0  &  \cdots & c_{n,1;1} & 0   \\

C_{2,1;2} & Q_{2\overline 1} \overline{\rm{det}(Q(\zeta))} &  C_{3,1;2}  &  Q_{2\overline 2} \overline{\rm{det}(Q(\zeta))} &  \cdots & C_{n+1,1;2}  & Q_{2\overline n} \overline{\rm{det}(Q(\zeta))}\\

-c_{1,1;2} &  0   &-c_{2,1;2}    &  0  &  \cdots & -c_{n,1;2} & 0  \\

\vdots  &  \vdots&   \vdots &  \vdots&  \cdots & \vdots &  \vdots \\

C_{2,1;n} & Q_{n\overline 1} \overline{\rm{det}(Q(\zeta))} &   C_{3,1;n}  & Q_{n\overline 2} \overline{\rm{det}(Q(\zeta))} &  \cdots & C_{n+1,1;n}  & Q_{n\overline n} \overline{\rm{det}(Q(\zeta))} \\
 
  \frac{c_{1,1;n} }{(-1)^{n+1}}& 0  &      \frac{c_{2,1;n}}{(-1)^{n+1}}  &  0  &  \cdots &     \frac{c_{n,1;n}}{(-1)^{n+1}} & 0  \\

\end{matrix}\right).
\end{equation*} 
The surjectivity of $\widetilde{L_2}$, and thus the one of $L_2$, is then a consequence of Theorem 2.1 \cite{be-de2} and of the following lemma:
\begin{lemma}\label{nonneg}
 The partial indices of $\overline {\widetilde{G_2}^{-1}}\widetilde{G_2}$, as defined in Subsection \ref{subrie}, are greater than or equal to $D_1-1\geq D_1-2.$
\end{lemma}
\begin{proof}[Proof of Lemma \ref{nonneg}]

Denote by $\kappa_{1} \geq \ldots \geq \kappa_{2n}$ the ordered partial indices of $\overline {\widetilde{G_2}^{-1}} \widetilde{G_2}$, and let $\Lambda(\cdot)$ be the diagonal matrix map with entries $\zeta^{\kappa_{1}}, \ldots, \zeta^{\kappa_{2n}}$. We know from Lemma 5.1 \cite{gl1} that  there exists a smooth map 
$\Theta: \overline{\Delta} \to GL_{2n}(\C)$, holomorphic on $\Delta$, such that 
\begin{equation}\label{eqfact}
-\Theta \overline {\widetilde{G_2}^{-1}}\widetilde{G_2} = \Lambda \overline \Theta.
\end{equation}
Let $l = (\lambda_{1},\mu_1, \ldots, \lambda_{n},\mu_{n})$ be the last row of the matrix $\Theta$. We get the following system of $2n$ equations
\begin{equation*}
\left\{
\begin{array}{lll} 
 \sum_{k=1}^n C_{j+1,1;k} \lambda_k +  \sum_{k=1}^n (-1)^{k+1}c_{j,1;k} \mu_k &=& -\overline {\rm det(Q)} \zeta^{\kappa_{2n}} \overline \lambda_{j} \\ 
\\
 \sum_{k=1}^n a_{2k-1,2j} \lambda_k +  \sum_{k=1}^n C_{1,k+1}^j \mu_k  &=&-\overline {\rm det(Q)}  \zeta^{\kappa_{2n}} \overline \mu_{j} \\
\end{array}
\right.
\end{equation*}
with $j=1,\ldots,n$.
Performing the column operations \eqref{eqop} on that system and focusing {\it only} on even rows gives
\begin{equation*}
\left\{
\begin{aligned}
 \overline{\rm{det}(Q(\zeta))}  \sum_{k=1}^n Q_{k\overline 1}\lambda_k &=-\overline {\rm det(Q(\zeta))} \zeta^{\kappa_{2n}}  \left( \overline \mu_{1} - \sum_{k=1}^{n} {S}_{\bar{k} \bar {1}} \overline \lambda_{k}\right) \\
& \vdots & \\
  \overline{\rm{det}(Q(\zeta))}  \sum_{k=1}^n Q_{k\overline n}\lambda_k &=-\overline {\rm det(Q(\zeta))} \zeta^{\kappa_{2n}}  \left( \overline \mu_{n} - \sum_{k=1}^{n} {S}_{\bar{k} \bar {n}} \overline \lambda_{k}\right) \\  
 \end{aligned}
\right.
\end{equation*}
Recall that by assumption $Q(\zeta)$ is invertible for all $\zeta\in \partial \Delta$, and thus dividing by $\zeta^{\kappa_{2n}}\overline{{\rm det}(Q(\zeta))}$, leads to 
\begin{equation}\label{eqsys}
\left\{
\begin{aligned}
\overline{\zeta}^{\kappa_{2n}}\sum_{k=1}^n Q_{k\overline 1}\lambda_k &=-    \overline \mu_{1} + \sum_{k=1}^{n} {S}_{\bar{k} \bar {1}} \overline \lambda_{k} \\
& \vdots & \\
\overline{\zeta}^{\kappa_{2n}} \sum_{k=1}^n Q_{k\overline n}\lambda_k &=-   \overline \mu_{n} +  \sum_{k=1}^{n} {S}_{\bar{k} \bar {n}} \overline \lambda_{k} \\  
 \end{aligned}
\right.
\end{equation}
Assume by contradiction that  $\kappa_{2n}\leq D_1-2$. Then the right hand side of each of the  above equations is antiholomorphic  since $S_{\bar{i} \bar{j}}$ are antiholomorphic polynomials, while the left hand side is holomorphic since $Q_{i\overline j}$ are holomorphic polynomials divisible by $\zeta^{D_1-1}$. They must then both vanish, leading to
$$Q\left(\begin{matrix} \lambda_1 \\ \vdots \\ \lambda_n \end{matrix}\right)=0$$
which implies that $\lambda_j$, $j=1\ldots,n$, is identically equal to zero.  This implies directly from \eqref{eqsys} that  $\mu_j$, $j=1\ldots,n$ also vanishes identically, contradicting the invertibility of the matrix $\Theta$. 
\end{proof}

This proves i. 

\vspace{0.5cm}

Finally, it remains to prove ii., that is, to find the dimension of the kernel of $\partial_2 F(\rho,\bm{f_0})$. 

\begin{lemma}\label{lemker}
The real dimension of the kernel of $\partial_2 F(\rho,\bm{f_0})$ is equal to ${\rm ind}\left(-\overline{Q^{-1}}Q\right)+d(2k_0+1)-2n(D_1-2)$.
\end{lemma}
\begin{proof}
The proof relies essentially on Theorem 2.4 \cite{be-de2}. Recall that the matrix $G$ is of the form  \eqref{eqG} with $G_2$ given by  \eqref{eqG_2}. We write 
 \begin{equation}\label{eqG2}
 G(\zeta)=\underbrace{\left(\begin{array}{cccccccccccc}
\frac{1}{2}I_d &  & (*)\\
   & G_2(\zeta) &  \\
 (0) &   &  -i\zeta^{k_0} I_d\\
\end{array}\right)}_{\tilde{G}(\zeta)} \times 
\underbrace{\left(\begin{array}{cccccccccccc}
I_d &  & (0)\\
   & D(\zeta) &  \\
 (0) &   &  I_d\\
\end{array}\right)}_{\tilde{D}(\zeta)} 
\end{equation}
and we note that the kernel of the operators $$\bm{f} \mapsto \partial_2 F(\rho,\bm{f_0})\bm{f}=2\Re e  \left[\overline{G(\zeta)}\bm{f}\right]$$ 
and 
$$\left(\mathcal A^{k,\alpha}_{0}\right)^d \times \left(\mathcal A^{k,\alpha}_{0^{D_1-1}} \right)^{2n}\times \left(\mathcal A^{k,\alpha}\right)^d  \ni \bm{f} \mapsto 2\Re e  \left[\overline{\tilde{G}(\zeta)}\bm{f}\right]$$
are of the same dimension. 
Now, using the same notation as in Theorem 2.4 \cite{be-de2}, we have $r=3$, $m_1=1$ and $N_1=d$,  $m_2=D_1-1$ and $N_2=2n$, and 
$m_3=0$ and $N_3=d$. Recall that the Maslov index $\kappa$ is given by the winding number at $0$ of the function 
$\zeta \mapsto \det\left(-\overline{\tilde{G}(\zeta)^{-1}}\tilde{G}(\zeta)\right)$. 
Since $\det(\widetilde{G_2})=(2i)^n\det(Q)$, we directly obtain 
$$\kappa= 0+{\rm ind} (Q) + 2dk_0={\rm ind} (Q) + 2dk_0.$$
Thus, the dimension of the kernel of $\partial_2 F(\rho,\bm{f_0})$ is equal to
\begin{eqnarray*}
\kappa+2n+2d-\sum_{j=1}^3N_jm_j&= &{\rm ind}\left(-\overline{Q^{-1}}Q\right)+2dk_0+2n+2d-d-2n(D_1-1)\\
&=&{\rm ind}\left(-\overline{Q^{-1}}Q\right)+d(2k_0+1)-2n(D_1-2).\\
\end{eqnarray*}
 
\end{proof}

This concludes the proof of the theorem.

\end{proof}

\begin{remark} In the paper \cite{be-bl-me}, the authors impose a pointwise constraint on the $\tilde{g}$ component of the lift, namely $\tilde{g}_\ell(1)=c_\ell/2$, $\ell=1,\ldots,d$. Imposing such a condition in the present setting would imply that the real dimension of the corresponding manifold of stationary lifts is ${\rm ind}\left(-\overline{Q^{-1}}Q\right)+2k_0d-2n(D_1-2)$. This follows  from the fact that the integer $m_3$ in the proof of the above Lemma  \ref{lemker}, would be equal to $1$. In view of Remark \ref{remnondeg}, if the model $M_H$  is strongly Levi nondegenerate, then $D_1=2$, $k_0=1$ and $Q(\zeta)=\zeta \sum c_jA_j$, which  implies ${\rm ind}\left(-\overline{Q^{-1}}Q\right)=2n$ and allows us to recover the dimension stated in Theorem 3.1 \cite{be-bl-me}.
\end{remark}

\subsection{Finite jet determination of generalized stationary lifts and open questions}
Let $M_H=\{\rho=0\}$ be an admissible model.  
For an integer $\ell_0 \in \{1,\ldots,k\}$, where $k$ is large enough, we consider the linear $\ell_0$-jet map at $\zeta=1$ defined on 
$(\mathcal{A}^{k,\alpha})^{2(n+d)}$ by 
 $$\mathfrak j_{\ell_0}({\bm f})=\left ( {\bm f}(1), {\bm f}'(1), \ldots, {\bm f}^{(\ell_0)}(1)\right )\in  \mathbb C^{(n+d)(\ell_0 + 1)}$$
 We prove the following  jet determination result for stationary lifts. 
 \begin{prop}\label{propjet}
Consider a stationary lift ${\bm f_0} \in \mathcal{S}_0^{k_0}(M_H)$ of the form \eqref{eqdisini} where $(c_1,\ldots,c_d)$ and $V$ form an admissible pair for $M_H$. Then there exist an open neighborhood $U$ of $\rho$ in $X$, $\varepsilon>0$, and an integer $\ell_0\geq 1$ such that for any $r \in U$ the jet map 
$\mathfrak j_{\ell_0}$ is injective on the submanifold $\{\bm{f} \in \mathcal{S}^{k_0}_0(\{r=0\})\ \ | \  
\|\bm{f}-\bm{f_0}\|<\varepsilon\}$.     
\end{prop}
Note that in the following proof we use the same notation as in the proof of Theorem \ref{theodiscs}. We also point out that the outline of the proof is now relatively standard and similar to the proofs of Lemma 5.3 \cite{be-de-la} and Proposition 3.8 \cite{be-me}.  
\begin{proof}
We set ${\ell_0}:={\rm max}\{2k_0,\kappa_{2n}\}$ where $\kappa_{2n}$ is the largest partial index of $-\overline {\widetilde{G_2}^{-1}}\widetilde{G_2}$.  Since $k$ is large enough, the restriction of the map $\mathfrak j_{\ell_0}$ to the submanifold $\mathcal{S}_0^{k_0}(M_H)$ is of class $\mathcal{C}^1$. We need to show that its differential 
$$\mathfrak j_{\ell_0}: T_{\bm f_0} \mathcal{S}^{k_0}_0(M_H) \to  \mathbb C^{(n+d)(\ell_0 + 1)}$$
is injective. Note that we have $T_{\bm f_0} \mathcal{S}^{k_0}_0(M_H)={\rm ker} \left(2\Re e \left[\overline{G}\cdot\right]\right)$. Let ${\bm f} \in T_{\bm f_0} \mathcal{S}^{k_0}_0(M_H)$ with a trivial $\ell_0$-jet, that is, $\mathfrak j_{\ell_0}(\bm f)=0$. This disc satisfies the following system of $2n+2d$ equations
\begin{equation}\label{eqker}
\overline{G}{\bm f} + G\overline{\bm f}=0.
\end{equation}
We first note that the last $d$ equations of \eqref{eqker} give 
\begin{equation}\label{eqind}
i\overline{\zeta^{k_0}} \tilde{g}(\zeta)-i\zeta^{k_0} \overline{\tilde{g}(\zeta)}=0,
\end{equation}
which directly implies  
$$\tilde{g}=\sum_{\ell=0}^{k_0-1}a_\ell \zeta^\ell+a_{k_0}\zeta^{k_0}+ \sum_{\ell=k_0+1}^{2k_0}\overline{a_{2k_0-\ell}} \zeta^\ell$$
with $a_\ell \in \C^d$ and $a_{k_0} \in \R^d$. Thus, $\tilde{g}$ is determined by its $2k_0$-jet and so $\tilde{g}\equiv 0$.
We now focus on the previous $2n$ equations in (\ref{eqker}) which, using the factorization \eqref{eqG_2fact}, are of the form
 \begin{equation*}
\overline{\widetilde{G_2}(\zeta)\tilde{D}(\zeta)}\underbrace{\left(\begin{matrix}
h_1  \\
 \tilde{h}_1\\
 \vdots \\
 h_n  \\
 \tilde{h}_n\\
\end{matrix}\right)}_{\bm h} 
+\widetilde{G_2}(\zeta)\tilde{D}(\zeta)\left(\begin{matrix}
\overline{h_1}  \\
 \overline{\tilde{h}_1}\\
 \vdots \\
 \overline{h_n}  \\
 \overline{\tilde{h}_n}\\
\end{matrix}\right) =0.
\end{equation*}
Using \eqref{eqfact}, we  write, permuting the components of  $\overline{\tilde{D}}{\bm h}$ if necessary, 
$$\Theta \overline{\tilde{D}}{\bm h} =-\Lambda \overline{\Theta \overline{\tilde{D}}{\bm h}}.$$
 It follows that the disc $\Theta \overline{\tilde{D}}{\bm h}$ satisfies equations similar to \eqref{eqind} and is then determined by its jet of order $\kappa_{2n}$. We note that since ${\bm h}$ has a trivial $\ell_0$-jet at $
\zeta=1$, the same applies to $\Theta \overline{\tilde{D}}{\bm h}$. It follows that $\Theta \overline{\tilde{D}}{\bm h}\equiv 0$ and so ${\bm h}\equiv 0$ since $\Theta 
\overline{\tilde{D}}$ is an isomorphism. 
Finally the first $d$ equations of \eqref{eqker} imply directly that $g \equiv 0$. This achieves the proof. 
\end{proof}

\vspace{0.5cm}

We end the paper with three open questions and directions related to the jet determination of CR mappings.  In \cite {be-bl-me} and \cite{be-me2},  the authors  obtain  a {\it filling property} that allows us to conclude that  $2$-jet determination  for sufficiently smooth 
CR automorphisms between Levi nondegenerate  submanifolds  holds.  We observe that, in the hypersurface case, the filling property  is  automatically satisfied for Levi nondegenerate  ones, as well as for polynomial models (see \cite{be-bl} and 
Lemma 5.4 \cite{be-de-la}), leading to finite jet determination.  At the moment, it is not clear whether the discs in the submanifold 
$\left\{\bm{f} \in \mathcal{S}^{k_0}_0(\{r=0\})\ \ | \  \|\bm{f}-\bm{f_0}\|<\varepsilon\right\}$ constructed in Theorem \ref{theodiscs} enjoy any filling property.  For a general overview of the finite jet determination problem, the reader is invited to consult the survey given in \cite{LM}.

\vspace{0.3cm}

\noindent {\bf Question 1.} In \cite{be-me}, we proved that the $1$-jet determination for stationary lifts for a nondegenerate submanifold  $M$ leads to a filling property in the conormal bundle $N^*M$ in $T^*\C^{n+d}$. In our context, we establish in Proposition \ref{propjet} a higher order jet determination for stationary lifts  and it would 
be interesting to understand whether this implies any filling property in the conormal bundle under additional assumptions on $M$. 

\vspace{0.3cm}

\noindent {\bf Question 2.}  Consider for a submanifold $M \subset \C^{n+d}$ given by \eqref{eqdeg} the evaluation map 
$\Psi: \mathcal{S}^{k_0}_0(M) \to \C^{n+d}$ defined by 
\begin{equation}\label{eqfill}
\Psi(\bm f)=f(0). 
\end{equation}
In case the model $M_H$ is strongly Levi nondegenerate with $A:=\sum_{\ell=1}^dc_\ell A_\ell$ invertible and $k_0=1$, the differential $d_{\bm f_0}\Psi: T_{\bm f_0} \mathcal{S}^1_0(M_H) \to \C^{n+d}$ at ${\bm f_0}$ given by \eqref{eqdisini}, is onto if and only if the matrix $\transp \overline{D}A^{-1}D$, where $D$ is the $n\times d$ matrix whose $\ell^{th}$ column is $A_\ell V$, is invertible (see \cite{be-bl-me,be-me2}). 
The question is then to find an analogous object in the degenerate setting. To this end, we focus once more on the strongly Levi nondegenerate case. We have already observed that $Q(\zeta)=\zeta A$ in Remark \ref{remnondeg}, and we note that  $D$ is contained in the matrix $G$ \eqref{eqG}. Indeed, $G$ can be further detailed 
\begin{equation}\label{eqG3}
G(\zeta)=\left(\begin{array}{ccccccccccccc}
\frac{1}{2}I_d & (1-\zeta)\transp D & 0 & 0\\
 0    & \zeta \transp A& I_n &   2(1-\zeta)D  \\
  0      & i\zeta \transp A & -iI_n  &  -2i(1-\zeta)D  \\
0 & 0   &  0  &  -i\zeta I_d\\
\end{array}\right),
\end{equation}
The matrix $D$ is essential in the description of the tangent space $T_{\bm f_0} \mathcal{S}^1_0(M_H)={\rm ker} \left(2\Re e [\overline{G}\cdot]\right)$, 
which is a key step in establishing the surjectivity of $d_{\bm f_0}\Psi$. In the degenerate setting, the transpose of the matrix
$$\left(\begin{matrix}
-P_{1,\overline{z_1}}\left((1-\zeta)V,(1-\overline{\zeta})\overline{V}\right)&  \ldots &  -P_{1,\overline{z_n}}\left((1-\zeta)V,(1-\overline{\zeta})\overline{V}\right)       \\
\vdots&  \ddots & \vdots       \\
-P_{d,\overline{z_1}}\left((1-\zeta)V,(1-\overline{\zeta})\overline{V}\right)&  0& -P_{d,\overline{z_n}}\left((1-\zeta)V,(1-\overline{\zeta})\overline{V}\right)    \\
\end{matrix}\right) $$
is a solid candidate for being an analogous to $D$, up to suitable powers of $(1-\zeta)$. In this vein, it would be interesting to first study  the case where the polynomials $P_1,\ldots,P_d$ share the same degree. Yet, at this stage, it seems that we are still far away from a complete description of 
the tangent space $T_{\bm f_0} \mathcal{S}^{k_0}_0(M_H)$. This is in part due to the fact that not much is known about the matrix $\Theta(\zeta)$ in the factorization \eqref{eqfact}, as opposed to the nondegenerate setting where 
$\Theta$ is explicity found  and turns out to be constant (see Equation $(2.12)$ in \cite{be-bl-me}).   
   
   \vspace{0.3cm}

\noindent {\bf Question 3.} As observed in \cite{be-me,tu,tu3}, nondefective stationary discs play an important role in the 
jet determination of CR mappings of strongly Levi nondegenerate
submanifolds. These discs are particularly relevant in the theory of wedge extension of CR maps since they enjoy an 
important filling property 
\cite{ber,ba-ro-tr,tu0}.
The nondefectivity of a stationary disc of the form \eqref{eqdisini} is algebraically characterized  in \cite{be-me}, again by 
means of the matrix $D$, and it would be  interesting to obtain a similar algebraic condition that characterizes 
the nondefectivity of such a disc for degenerate submanifolds. We address this question in a forthcoming paper.

\vspace{0.5cm}

\noindent  {\it Acknowledgments.}  This work originates from the  Summer Research Camp in Mathematics (2024) designed by the Department of Mathematics at the American University of Beirut (AUB), and that benefitted from a generous support from the Center for 
Advanced Mathematical Sciences and the Faculty of Arts and Sciences at AUB.

\vskip 1cm

{\small
\noindent Mohammad Tarek Al Masri, Florian Bertrand, Lea Oueidat, Hadi Zoghaib\\
Department of Mathematics\\\
American University of Beirut, Beirut, Lebanon\\{\sl E-mail addresses}: maa369@mail.aub.edu, fb31@aub.edu.lb, lea.oueidat@outlook.com, hsz08@mail.aub.edu\\

\noindent Francine Meylan \\
Department of Mathematics\\
University of Fribourg, CH 1700 Perolles, Fribourg\\
{\sl E-mail address}: francine.meylan@unifr.ch\\
}


\begin{thebibliography}{11111}

\bibitem{al-be-mc-ou-zo}  M. T. Al Masri, F. Bertrand, J. Mchaimech, L. Oueidat, H. Zoghaib, 
{\it Construction of stationary discs for perturbations of decoupled submanifolds in $\C^4$}, to appear in Arch. Math. (Brno). 



\bibitem{ber}M.S. Baouendi, P. Ebenfelt, L.P. Rothschild, {\it Real submanifolds in complex space and their mappings}, 
Princeton Mathematical Series, {\bf 47}. Princeton University Press, Princeton, NJ, 1999. xii+404 pp.

\bibitem{ba-ro-tr}  M.S. Baouendi, L.P. Rothschild, J.-M. Tr\' epreau,   
 \textit{ On the geometry of analytic discs attached to real manifolds}, 
J. Differential Geom.  \textbf{39}  (1994),   379-405.  


\bibitem{be-bl} F. Bertrand, L. Blanc-Centi, {\it Stationary holomorphic discs and finite jet determination problems},  
Math. Ann. {\bf 358} (2014), 477-509.

\bibitem{be-bl-me} F. Bertrand, L. Blanc-Centi, F. Meylan, {\it Stationary discs and finite jet determination for non-degenerate generic real submanifolds},
Adv. Math. {\bf 343} (2019), 910-934. 

\bibitem{be-de1} F. Bertrand, G. Della Sala, {\it  Stationary discs for smooth hypersurfaces of finite type and finite jet determination}, J. Geom. Anal. {\bf 25}  (2015), 2516-2545.


\bibitem{be-de2} F. Bertrand, G. Della Sala, {\it Riemann-Hilbert problems with constraints}, Proc. Amer. Math. Soc. {\bf 147} (2019),  2123-2131.


\bibitem{be-de-la} F. Bertrand, G. Della Sala, B. Lamel, {\it Jet determination of smooth CR automorphisms and generalized stationary discs}, Math. Z. {\bf 294}  (2020), 1611-1634.


\bibitem{be-me} F. Bertrand, F. Meylan, {\it Nondefective stationary discs and $2$-jet  determination in higher codimension},  J. Geom. Anal. {\bf 31} (2021), 6292-6306.

\bibitem{be-me2} F. Bertrand, F. Meylan, {\it  Explicit construction of stationary discs and its consequences for nondegenerate quadrics},  Ann. Sc. Norm. Super. Pisa Cl. Sci. (5) {\bf 24} (2023), 1875-1892.



\bibitem{bl} L. Blanc-Centi, {\it Stationary discs glued to a Levi non-degenerate hypersurface}, Trans. Amer. Math. Soc. {\bf 361} (2009), 3223-3239. 

\bibitem{bl-me} L. Blanc-Centi, F. Meylan,  {\it Chern-Moser operators and weighted jet determination problems in higher codimension},
Internat. J. Math. {\bf 33} (2022), no. 6, Paper No. 2250045, 15 pp.

\bibitem{BG} T. Bloom, I. Graham, {\it On "type"  conditions for generic real submanifolds of $\C^n$},  
Invent. Math. {\bf 40}  (1977), 217-243.





\bibitem{fo} F. Forstneri\v{c}, {\it Analytic disks with boundaries in a maximal real submanifold of $\C^2$}, Ann. Inst. Fourier {\bf 37} (1987), 1-44.


\bibitem{gl1}J. Globevnik, {\it Perturbation by analytic discs along maximal real submanifolds of $\C^N$},
 Math. Z. {\bf 217} (1994), 287-316.

\bibitem{gl2}J. Globevnik, {\it Perturbing analytic discs attached to maximal real submanifolds of $\C^N$}, 
Indag. Math.  {\bf 7} (1996), 37-46.



 
 
 
\bibitem{hi-ta}C.D. Hill, G. Taiani, {\it Families of analytic discs in $\C^n$ with boundaries on a prescribed CR submanifold},
 Ann. Scuola Norm. Sup. Pisa Cl. Sci. (4) {\bf 5} (1978), 327-380.
 
 
 \bibitem{hu} X. Huang, {\it A non-degeneracy property of extremal mappings and iterates of holomorphic 
self-mappings}, Ann. Scuola Norm. Sup. Pisa Cl. Sci. (4) {\bf 21} (1994), 399-419.

 

 \bibitem{LM} B. Lamel, N. Mir, {\it Two decades of finite jet determination of CR mappings,} Complex Anal. Synerg. {\bf 8} (2022), no. 4, 
 Paper No. 19, 15 pp.
 
 \bibitem{K} J.\ J. Kohn, \textit{Boundary behavior of
	$\delta$ on weakly pseudo-convex manifolds of dimension
	two}, J. Differential Geometry {\bf 6} (1972), 523-542.
 

  \bibitem{le} L. Lempert, {\it La m\'etrique de Kobayashi et la repr\'esentation des domaines sur la boule}, Bull. Soc. Math. France 
{\bf 109} (1981), 427-474.







\bibitem{tu0}  A. Tumanov,  {\it Extension of CR-functions into a wedge from a manifold of finite type}, (Russian) Mat. Sb. (N.S.) {\bf 136 (178)} (1988), 128-139; translation in Math. USSR-Sb. {\bf 64} (1989), 129-140.


\bibitem{tu} A. Tumanov, {\it Extremal discs and the regularity of CR mappings in higher codimension}, Amer. J. Math. 
{\bf 123} (2001), 445-473.

\bibitem{tu3} A. Tumanov, {\it Stationary Discs and finite jet determination for CR mappings in higher codimension},   Adv. Math. {\bf 371} (2020), 107254, 11 pp.

\bibitem{ve}N.P. Vekua, {\it Systems of singular integral equations}, Noordhoff, Groningen (1967) 216 pp.


\bibitem{we} S. M. Webster, {\it Analytic discs and the regularity of C-R mappings of real submanifolds in $\C^n$.} Complex analysis of several variables (Madison, Wis., 1982), 199-208, Proc. Sympos. Pure Math., 41, Amer. Math. Soc., Providence, RI, 1984.



\end{thebibliography}
\end{document}